\documentclass[11pt, a4paper]{article}

\usepackage{amsmath}
\usepackage{amssymb}
\usepackage{amsthm}
\usepackage{appendix}

\newtheorem{theorem}{Theorem}[section]

\newtheorem{lemma}{Lemma}[section]
\newtheorem{remark}{Remark}[section]

\newtheorem*{thma}{Theorem A}
\newtheorem*{thmb}{Theorem B}
\newtheorem*{thmc}{Theorem C}
\newtheorem*{propa}{Proposition A}

\begin{document}
\title{On Markushevich bases $\{x^{\lambda_n}\}_{n=1}^{\infty}$ for their closed span in weighted $L^2 (A)$ spaces
over sets $A\subset [0,\infty)$ of positive Lebesgue measure, Hereditary completeness, and Moment problems}
\author{Elias Zikkos\\
Department of Mathematics, Khalifa University of Science and Technology,\\
Abu Dhabi, United Arab Emirates\\
email address:  elias.zikkos@ku.ac.ae and eliaszikkos@yahoo.com}

\maketitle

\begin{abstract}

Inspired by the work of Borwein and Erdelyi \cite{BE1997JAMS} on generalizations of M\"{u}ntz's theorem,
we investigate the properties of the system $\{x^{\lambda_n}\}_{n=1}^{\infty}$ in weighted $L^p (A)$ spaces, for $p\ge 1$,
denoted by $L^p_w (A)$, where

(I) $A$ is a measurable subset of the real half-line $[0,\infty)$ having positive Lebesgue measure,

(II) $w$ is a non-negative integrable function defined on $A$, and

(III) $\{\lambda_n\}_{n=1}^{\infty}$ is a strictly increasing sequence of positive real numbers such that

$\inf\{\lambda_{n+1}-\lambda_n \}>0$ and $\sum_{n=1}^{\infty}\lambda_n^{-1}<\infty$.

Based on the  $\bf{Remez-type\,\, inequality}$ (\cite[Theorem 5.6]{BE1997JAMS}),
we find a sharp lower bound for the $distance$ between the function $x^{\lambda_n}$ and the closed span
of the system  $\{x^{\lambda_k}\}_{k\not= n}$ in the $L^p_w (A)$ spaces. Then intrigued by the
``Clarkson-Erd\H{o}s-Schwartz Phenomenon'' regarding the closed span of an incomplete system $\{x^{\lambda_n}\}_{n=1}^{\infty}$
in the $L^p_w (A)$ spaces (\cite[Theorem 6.4]{BE1997JAMS}), we prove that a function $f$ in
$\overline{\text{span}}\{x^{\lambda_n}\}_{n=1}^{\infty}$ in the Hilbert space $L^2_w (A)$, admits the $\bf{Fourier-type}$ series representation
$f(x)=\sum_{n=1}^{\infty} \langle f, r_n\rangle_{w,A} x^{\lambda_n}$ a.e on $A$, where $\{r_n\}_{n=1}^{\infty}$
is the unique biorthogonal family of $\{x^{\lambda_n}\}_{n=1}^{\infty}$ in $\overline{\text{span}}\{x^{\lambda_n}\}_{n=1}^{\infty}$
in $L^2_w (A)$. As a result, we show that the system $\{x^{\lambda_n}\}_{n=1}^{\infty}$ is a $\bf{Markushevich\,\, basis}$
for $\overline{\text{span}}\{x^{\lambda_n}\}_{n=1}^{\infty}$ in $L^2_w (A)$.
Furthermore, we consider a $\bf{moment\,\, problem}$: assuming certain growth conditions on a sequence
$\{d_n\}_{n=1}^{\infty}$ of real numbers, we find a function $f\in\overline{\text{span}}\{x^{\lambda_n}\}_{n=1}^{\infty}$
in $L^2_w (A)$, serving as a solution to
\[
\int_{A} f(x)\cdot  x^{\lambda_n}\cdot w(x)\, dx=d_n, \qquad n=1,2,\cdots
\]
Finally, if $m\le w(x)\le M$ on $A$ for some positive numbers $m$ and $M$ and the set $A$ contains an interval $[a, r_A]$,
where $a\ge 0$ and $r_A$ is the essential supremum of $A$, we prove that the system
$\{x^{\lambda_n}\}_{n=1}^{\infty}$ is $\bf{hereditarily\,\, complete}$ in $\overline{\text{span}}\{x^{\lambda_n}\}_{n=1}^{\infty}$
in the space $L^2_w(A)$. As a result, a general class of compact operators on the closure is constructed that admit spectral synthesis.
\end{abstract}



\section{Introduction and the Main Results}
\setcounter{equation}{0}

This work has been motivated from the Borwein-Erdelyi results \cite{BE1997JAMS} on generalizations
of the classical M\"{u}ntz-Sz\'{a}sz theorem which in turn answered a question posed by S. N. Bernstein
regarding the completeness of systems $\{x^{\lambda_n}\}_{n=1}^{\infty}$ in the spaces $C_0[0,1]$ and $L^p (0,1)$ for $p\ge 1$,
when $\Lambda=\{\lambda_n\}_{n=1}^{\infty}$ is a strictly increasing sequence of positive real numbers diverging to infinity.
By the M\"{u}ntz-Sz\'{a}sz theorem the span of the system $\{x^{\lambda_n}\}_{n=1}^{\infty}$ is dense in the spaces $C_0[0,1]$ and $L^p (0,1)$
for $p\ge 1$,  if and only $\sum_{n=1}^{\infty}\lambda_n^{-1}=\infty$.

Contributions to the M\"{u}ntz-Sz\'{a}sz problem were given later on
by J. A. Clarkson, P. Erd\H{o}s, L. Schwartz,
W. A. J. Luxemburg, J. Korevaar, P. Borwein,  T. Erdelyi, W. B. Johnson (see \cite{CE,LK,BE,BE1997JAMS})
where the interval $[0,1]$ is replaced by an interval $[a,b]$ away from the origin for $a>0$
or even by a compact set $K\subset [0,\infty)$  of positive Lebesgue measure. The reader may consult
relevant survey articles and books such as \cite{Pinkus,Almira,BE,Gurariy}. There is still an ongoing research
on M\"{u}ntz-Sz\'{a}sz type problems
(see \cite{Agler1,Agler2,Chalendar,Noor,Ait-Haddou,Lefevre1,Lefevre2,Lefevre3,Fricain,Jaming,Trefethen,Zikkos2011JAT,Zikkos2024JMAA,Zikkos2025Ideal}).

When $\sum_{n=1}^{\infty}\lambda_n^{-1}<\infty$ the closed span of the system $\{x^{\lambda_n}\}_{n=1}^{\infty}$
is a proper subspace of $L^p (0,1)$. In this case we have the ``Clarkson-Erd\H{o}s-Schwartz Phenomenon'':
any function $f$ $\overline{\text{span}}\{x^{\lambda_n}\}_{n=1}^{\infty}$ in $L^p (0,1)$, is extended to an analytic function
throughout the interior of the slit disk
$\mathbb{D}^*:= \{z:\, |z|<1\}\setminus (-1,0)$ admitting a series representation
of the form $f(z)=\sum_{n=0}^{\infty} a_n z^{\lambda_n}$ which converges uniformly on compacta
(see \cite{CE}, \cite[Corollary 6.2.4]{Gurariy} and \cite[Theorem 8.2]{LK}).
A converse result appears in \cite{Korevaar1947}, \cite[Corollary 6.2.4]{Gurariy} and \cite[Theorem 8.2]{LK}.
We also note that in \cite{Dzh,Zikkos2011JAT,Zikkos2025Ideal} one finds results on the closed span in
$L^2 (-\infty, 0)$ and in $L^2 (a,b)$ of exponential systems of the form
\[
\{x^ke^{\lambda_n x}:\,\, n\in\mathbb{N}, \,\, k=0,1,2,\dots,\mu_n-1\}.
\]

A beautiful generalization of the ``Clarkson-Erd\H{o}s-Schwartz Phenomenon'' was given by Borwein and Erdelyi in
\cite[Theorems 6.1 and 6.4]{BE1997JAMS} where the interval $(0,1)$ is replaced by a $\bf{measurable}$ subset
of the real half-line $[0,\infty)$ having positive Lebesgue measure.
The crucial tool was a Remez-type inequality \cite[Theorems 5.1 and 5.6]{BE1997JAMS}
which is recalled below with several other results form \cite{BE1997JAMS} and right after we state ours.
But first we need to introduce some notation and definitions following the work of \cite{BE1997JAMS}.

\subsection{Notations and definitions}

Let $A\subset [0,\infty)$  be a measurable subset of the real half-line $[0,\infty)$ having positive Lebesgue measure
and let $r_A$ be the essential supremum of $A$, that is
\begin{equation}\label{essentialsup}
r_{A}:=\sup\{x\in [0,\infty): m(A\cap [x, \infty)>0\},\qquad \text{where $m$ is the Lebesgue measure}.
\end{equation}

Let $w$ be a real-valued non-negative integrable function defined on $A$, with $\int_{A} w(x)\, dx<\infty$, and such that
\begin{equation}\label{measurezero}
m(\{x\in A : w(x)=0\})=0.
\end{equation}
Let
\begin{equation}\label{rw}
r_{w}:=\sup\left\{x\in [0,\infty):\int_{A\cap(x,\infty)}w(t)\, dt>0\right\}.
\end{equation}
\begin{remark}\label{equalr}
Clearly, it always holds that $r_w\le r_A$. Also, it is easy to see that
if there is a positive number $m$ so that $m\le w(x)$ for all $x\in A$, then $r_w=r_A$.
\end{remark}
We also consider the slit disk
\begin{equation}\label{slit}
D_{r_w}:=\{z\in\mathbb{C}\setminus (-\infty, 0]: \,\, |z|<r_w\}.
\end{equation}

For $p\ge 1$, we denote by $L^p_w (A)$ the space of
$\bf real-valued$ measurable functions defined on $A$ such that $\int_{A}|f(x)|^p\cdot w(x)\, dx<\infty$, equipped with the norm
\[
||f||_{L^p_w (A)}:=\left(\int_{A}|f(x)|^p\cdot w(x)\, dx\right)^{\frac{1}{p}}.
\]
The space $L^2_w(A)$ is a real Hilbert space once endowed with the inner product
\[
\langle f,g \rangle_{w,A}:=\int_{A} f(x) g(x)\cdot w(x)\, dx.
\]
If $w\equiv 1$, then $L^p (A)$ is the space of functions such that $\int_{A}|f(x)|^p\, dx<\infty$ and we let
\[
||f||_{L^p(A)}:=\left(\int_{A}|f(x)|^p\, dx\right)^{\frac{1}{p}}\qquad\text{and}\qquad
\langle f,g \rangle_{A}:=\int_{A} f(x) g(x)\, dx.
\]
Given a weight $w$, if $A$ is an interval $[a,b]$ we use the notations
\[
||f||_{L^p_w ([a,b])}:=\left(\int_{a}^{b}|f(x)|^p\cdot w(x)\, dx\right)^{\frac{1}{p}}\qquad\text{and}\qquad
\langle f,g \rangle_{w,[a,b]}:=\int_{a}^{b} f(x) g(x)\cdot w(x)\, dx,
\]
and in addition, if $w\equiv 1$, we use the notations
\[
||f||_{L^p([a,b])}:=\left(\int_{a}^{b}|f(x)|^p\, dx\right)^{\frac{1}{p}}\qquad\text{and}\qquad
\langle f,g \rangle_{[a,b]}:=\int_{a}^{b} f(x) g(x)\, dx.
\]

\begin{remark}
We also use the notations $||f||_{A} : =\sup_{x\in A} |f(x)|$ and $||f||_{[a,b]} : = \sup_{x\in [a,b]} |f(x)|$.
\end{remark}

Throughout this article, in our results $\Lambda:=\{\lambda_n\}_{n=1}^{\infty}$ is
a strictly increasing sequence of positive real numbers so that
\begin{equation}\label{LKcondition}
\sum_{n=1}^{\infty}\frac{1}{\lambda_n}<\infty
\end{equation}
satisfying the gap condition
\begin{equation}\label{gapcondition}
\inf_{n\in\mathbb{N}}(\lambda_{n+1}-\lambda_n)>0.
\end{equation}
We associate to $\Lambda$ the system
\begin{equation}\label{system}
M_{\Lambda}:=\{e_{n}:\,\, n\in\mathbb{N}\}\qquad \text{where}\qquad e_{n}(x):=x^{\lambda_n}.
\end{equation}
\begin{remark}

\quad

(1) We denote by $\text{span}(M_{\Lambda})$ the set of all
finite linear combinations of elements from $M_{\Lambda}$ with $\bf real$ coefficients.

(2) We say that a function $f:A\mapsto \mathbb{R}$ belongs to the closed span of the system $M_{\Lambda}$ in $L^p_w(A)$,
if for every $\epsilon>0$ there is some $g_{\epsilon}\in \text{span}(M_{\Lambda})$ so that $||f-g_{\epsilon}||_{L^p_w(A)}<\epsilon$.
\end{remark}

\subsection{Borwein-Erdelyi results}

We state below several important results from \cite{BE1997JAMS} on Remez-type inequalities, extensions of the ``Clarkson-Erd\H{o}s-Schwartz Phenomenon'', and M\"{u}ntz-Sz\'{a}sz type theorems.

\subsubsection{$\bf{Remez-type\,\, inequalities}$}

Inequality $(\ref{Erdelyic})$ is considered by Borwein and Erdelyi to be the central result of their work while
inequality $(\ref{Erdelyi})$ is the $L^p$ version.

\begin{thma}\cite[Theorems 5.1 and 5.6]{BE1997JAMS}
Suppose a sequence $\Lambda$ satisfies the condition $(\ref{LKcondition})$.
Let $s>0$ and $p\in (0,\infty)$. Then the following are true.

(I) There exists a constant $c$ depending only on $\Lambda$ and $s$,
(and not on $\rho$, $A$, or the ``length'' of $f$) so that
\begin{equation}\label{Erdelyic}
||f||_{[0,\rho]}\le c\cdot||f||_{A},
\end{equation}
for every $f\in \text{span}(M_{\Lambda})$ and for every set $A\subset [\rho, 1]$ of Lebesgue measure at least $s$.

(II) There exists a constant $c$ depending only on $\Lambda$, $s$ and $p$,
(and not on $\rho$, $A$, or the ``length'' of $f$) so that
\begin{equation}\label{Erdelyi}
||f||_{[0,\rho]}\le c\cdot||f||_{L^p(A)},
\end{equation}
for every $f\in \text{span}(M_{\Lambda})$ and for every set $A\subset [\rho, 1]$ of Lebesgue measure at least $s$.
\end{thma}

\subsubsection{On the ``Clarkson-Erd\H{o}s-Schwartz Phenomenon''}

Based on Theorem $A$, Borwein and Erdelyi obtained the following
result for the closed span of an incomplete system $M_{\Lambda}$ in the $L^p_w (A)$ spaces.
Apart from condition $(\ref{LKcondition})$, the gap condition $(\ref{gapcondition})$ is also imposed this time.

\begin{thmb}\cite[Theorem 6.4]{BE1997JAMS}.
Let $A$  be a measurable subset of the real half-line $[0,\infty)$ of positive Lebesgue measure
and let $w$ be a non-negative integrable function defined on $A$, with $r_w$ as in $(\ref{rw})$.
Let $\Lambda=\{\lambda_n\}_{n=1}^{\infty}$ be a sequence of distinct positive real numbers that
satisfies $(\ref{LKcondition})-(\ref{gapcondition})$. Then the span of $\{x^{\lambda_n}\}_{n=1}^{\infty}$ is not dense in $L^p_w (A)$,
and every function $f\in L^p_w (A)$ belonging to the $L^p_w (A)$ closure of $\text{span}(M_{\Lambda})$ extends analytically
in the slit disk $D_{r_w}$ $(\ref{slit})$ and can be represented as
\begin{equation}\label{coeff}
f(x)=\sum_{n=1}^{\infty} a_n x^{\lambda_n}, \qquad x\in A\cap [0, r_w).
\end{equation}
\end{thmb}

\subsubsection{M\"{u}ntz-Sz\'{a}sz results on $L^p_w(A)$}

From Theorem $B$, Borwein and Erdelyi extended the M\"{u}ntz-Sz\'{a}sz theorem as follows.
\begin{thmc}\cite[Theorem 6.5]{BE1997JAMS}.
Let $A$  be a measurable subset of the real half-line $[0,\infty)$ of positive Lebesgue measure.
Let $\Lambda=\{\lambda_n\}_{n=1}^{\infty}$ be a strictly increasing sequence of positive real numbers, diverging to infinity.
Let $w$ be a non-negative integrable function defined on $A$. Then 
span of $\{x^{\lambda_n}\}_{n=1}^{\infty}$ is dense in $L^p_w (A)$ if and only if $\sum_{n=1}^{\infty}\lambda_n^{-1}=\infty$.
\end{thmc}

\subsection{Our Results}

Our first result, Theorem $\ref{Distances}$, is obtained from the Remez-type inequality $(\ref{Erdelyi})$
and we will refer to it as the $Distance$ result.
It plays a crucial role in the derivation of our other theorems.

Consider a sequence $\Lambda=\{\lambda_n\}_{n=1}^{\infty}$ that satisfies conditions $(\ref{LKcondition})-(\ref{gapcondition})$ and from
the system $M_{\Lambda}$ $(\ref{system})$ exclude an element $e_{n}(x)=x^{\lambda_n }$.
The resulting system is denoted by $M_{\Lambda_{n}}$, that is
\begin{equation}\label{exceptone}
M_{\Lambda_{n}}:=M_{\Lambda}\setminus\{e_{n}\}.
\end{equation}
Let $D_{A,w,p,n}$ stand for the $\bf distance$ between $e_{n}$
and the closed span of $M_{\Lambda_{n}}$ in $L^p_w(A)$, that is
\[
D_{A,w,p,n}:=\inf_{g\in \overline{\text{span}} (M_{\Lambda_{n}})} ||e_{n}-g||_{L^p_w(A)}.
\]
We will derive the sharp lower bound $(\ref{distancelowerbounds})$ for $D_{A,w,p,n}$.
\begin{theorem}\label{Distances}
Let $A$  be a measurable subset of the real half-line $[0,\infty)$ of positive Lebesgue measure
and let $w$ be a  real-valued non-negative integrable function defined on $A$, with $r_w$ as in $(\ref{rw})$.
Let $\Lambda=\{\lambda_n\}_{n=1}^{\infty}$ be a sequence of distinct positive real numbers that
satisfies $(\ref{LKcondition})-(\ref{gapcondition})$.
Then, for every $\epsilon>0$ there is a constant $u_{\epsilon}>0$, independent of $p\ge 1$ and $n\in\mathbb{N}$,
but depending on $\Lambda$, $A$ and $w$, so that
\begin{equation}\label{distancelowerbounds}
D_{A,w,p,n}\ge u_{\epsilon}(r_w-\epsilon)^{\lambda_n}.
\end{equation}
\end{theorem}

We then revisit Theorem $B$ and give some more details regarding the coefficients $a_n$ appearing in $(\ref{coeff}$).
If $P_j(x)=\sum_{n=1}^{p(j)} a_{j,n} x^{\lambda_n}$, for $j=1,2,\dots,$ are the functions in the span of $M_{\Lambda}$ such that $||f-P_j||_{L^p_{w,A}}\to 0$ as $j\to\infty$, we show that for each fixed $n\in\mathbb{N}$,
the $a_{j,n}$ coefficients of the $P_j$'s tend to the $a_n$ coefficient of $(\ref{coeff})$ as $j\to\infty$, see $(\ref{important})$.
In turn, this result is essential for deriving Lemma $\ref{span}$.

\begin{lemma}\label{revisiting}
Consider all the assumptions of Theorem B. Let
$f\in L^p_w (A)$ belong to the $L^p_w (A)$ closure of $\text{span}(M_{\Lambda})$, thus there exists a sequence
$\{P_j\}_{j=1}^{\infty}$ in $\text{span}(M_{\Lambda})$ where $P_j(x)=\sum_{n=1}^{p(j)} a_{j,n} x^{\lambda_n}$,
such that $||f-P_j||_{L^p_{w,A}}\to 0$ as $j\to\infty$. Then $f$ extends analytically
in the slit disk $D_{r_w}$ $(\ref{slit})$ and can be represented as
\begin{equation}\label{analyticextension}
f(z)=\sum_{n=1}^{\infty} a_n z^{\lambda_n}, \qquad z\in D_{r_w},\qquad
\end{equation}
with the series converging uniformly on compact subsets of $D_{r_w}$ and such that
\begin{equation}\label{important}
a_n=\lim_{j\to\infty} a_{j,n}\qquad \text{for}\quad n=1,2,\dots.
\end{equation}
\end{lemma}

We will show below that more is true regarding these $a_n$ coefficients in the case of the Hilbert space $L^2_w (A)$.
First we note that by $(\ref{distancelowerbounds})$ the distances $D_{A,w,2,n}$ are positive
thus $M_{\Lambda}$ is what we call a $minimal$ system in its closed span in $L^2_w (A)$. It is known that
a family $\{f_n\}_{n\in \mathbb{N}}$ of functions in a separable Hilbert space $\cal{H}$ endowed with an inner product
$\langle \cdot\, , \cdot \rangle_{\cal{H}}$ is minimal,
if and only if it has a $biorthogonal$ sequence in $\cal{H}$ (see \cite[Lemma 3.3.1]{Christensen}): in other words,
there exists a sequence $\{g_n\}_{n\in \mathbb{N}}\subset\cal{H}$ so that
\[
\langle g_n, f_m\rangle_{\cal{H}} =\begin{cases} 1, & m=n, \\  0, & m\not=n.\end{cases}
\]
A family $\{f_n\}$ in $\cal{H}$ is said to be complete if its closed span in $\cal{H}$ is equal to $\cal{H}$.
If a family is both complete and minimal then it has a unique biorthogonal family
in $\cal H$, however this family is not necessarily complete in $\cal H$ (see \cite{YoungP}).

Hence, assuming conditions $(\ref{LKcondition})-(\ref{gapcondition})$, the system $M_{\Lambda}$ has a unique biorthogonal family
in its closed span in $L^2_w (A)$. We denote this family by $r_{\Lambda}$
and we will show that $r_{\Lambda}$ is complete in that closure, thus
$M_{\Lambda}$ and $r_{\Lambda}$ are $Markushevich\,\, bases$ for the closed span of $M_{\Lambda}$ in $L^2_w (A)$.
To achieve this, we first obtain the sharp upper bound $(\ref{rnkbound})$ for the norms of the elements in the family $r_{\Lambda}$,
and then we derive the Fourier-type series representation $(\ref{representationf})$ for functions in the closed span of $M_{\Lambda}$
in $L^2_w (A)$. Regarding $(\ref{representationf})$, we
show that for functions $f$ in the closed span of $M_{\Lambda}$ in $L^2_w (A)$, the coefficient $a_n$ in the series expansion $f(x)=\sum_{n=1}^{\infty} a_n x^{\lambda_n}$ as in $(\ref{analyticextension})$,
is equal to the inner product $\langle f, r_{n} \rangle_{w, A}$.

\begin{theorem}\label{biorthogonalsystem}
Let $A$  be a measurable subset of the real half-line $[0,\infty)$ of positive Lebesgue measure
and let $w$ be a  real-valued non-negative integrable function defined on $A$, with $r_w$ as in $(\ref{rw})$.
Let $\Lambda=\{\lambda_n\}_{n=1}^{\infty}$ be a sequence of distinct positive real numbers that
satisfies $(\ref{LKcondition})-(\ref{gapcondition})$.
Then there exists a family of functions
\[
r_{\Lambda}=\{r_{n}:\,\, n\in\mathbb{N}\}\subset L^2_w(A)
\]
so that it is the unique biorthogonal sequence to $M_{\Lambda}$ in $L^2_w(A)$ which belongs to the closed span
of the system $M_{\Lambda}$ in $L^2_w(A)$, and the following are true:

$(I)$ For every $\epsilon>0$ there is a constant $m_{\epsilon}>0$, independent of $n\in\mathbb{N}$,
but depending on $\Lambda$, $A$ and $w$, so that
\begin{equation}\label{rnkbound}
||r_{n}||_{L^2_w (A)} \le  m_{\epsilon}(r_w-\epsilon)^{-\lambda_n},\qquad \forall\,\, n\in\mathbb{N}.
\end{equation}

$(II)$ Each function $f$ in the closed span of the system $M_{\Lambda}$ in $L^2_w (A)$,
extends analytically in the slit disk $D_{r_w}$ $(\ref{slit})$,  so that $f$ admits the Fourier-type series representation
\begin{equation}\label{representationf}
f(z)=\sum_{n=1}^{\infty}\langle f, r_{n} \rangle_{w, A} z^{\lambda_n},
\end{equation}
converging uniformly on compact subsets of the slit disk $D_{r_w}$.

$(III)$ For each function $f\in L^2_w(A)$, its associated series
\begin{equation}\label{representationffourier}
f^*(z):=\sum_{n=1}^{\infty} \langle f, r_{n} \rangle_{w, A} z^{\lambda_n},
\end{equation}
is analytic in $D_{r_w}$ and $f^*$ belongs to the closed span of the system $M_{\Lambda}$ in $L^2_w (A)$.

$(IV)$ The system $M_{\Lambda}$ is a Markushevich basis in its closed span in $L^2_w (A)$,
that is,
\[
\overline{\text{span}}(r_{\Lambda})=\overline{\text{span}}(M_{\Lambda})\quad \text{in}\quad L^2_w(A).
\]
\end{theorem}

\begin{remark}
It follows from $(III)$ above that for $f(x)=1$ on $A$, the series
\[
\sum_{n=1}^{\infty} \langle 1, r_{n} \rangle_{w, A} z^{\lambda_n},
\]
is analytic in $D_{r_w}$ and belongs to the closed span of the system $M_{\Lambda}$ in $L^2_w (A)$.
\end{remark}

Next, now that we know that the families $M_{\Lambda}$ and $r_{\Lambda}$ are Markushevich bases for the
closed span of $M_{\Lambda}$ in $L^2_w (A)$, we ask if they are also
\[
\text{Strong Markushevich Bases}
\]
for this closure.
This means whether for any disjoint union of two sets $N_1$ and $N_2$, such that
$\mathbb{N}=N_1\cup N_2$, the closed span of the mixed system
\[
\{x^{\lambda_n}:\,\, n\in N_1\}\cup\{r_{n}:\,\, n\in N_2\}
\]
is equal to the closed span of $M_{\Lambda}$ in $L^2_w (A)$. Another term used instead of the phrase ``strong Markushevich basis''
is the phrase ``$hereditarily\,\, complete\,\, system$''.

\smallskip

\begin{remark}
The answer to the above question is $affirmative$ in the special case
when $m\le w(x)\le M$ on $A$ for some positive numbers $m$ and $M$ and the set $A$ contains an interval $[a,r_A]$ for some $0\le a<r_A$
where $r_A$ is the essential supremum of $A$ (see Theorem $\ref{hereditarycompleteness}$).
\end{remark}

\begin{remark}
It would be very interesting to investigate also the case when
$A$ does not contain such an interval.
\end{remark}

Hereditary completeness of exponential systems $\{e^{i\lambda_n t}:\,\, n\in \mathbb{N}\}$ in $L^2(-a,a)$
has been studied in \cite{Baranov2013} where it was proved that a complete and minimal exponential system in $L^2(-a,a)$
is hereditarily complete up to one-dimensional effect (\cite[Theorem 1.1]{Baranov2013}). In \cite[Theorem 1.1]{Zikkos2024JMAA}
we proved that if $\Lambda$ satisfies the conditions $(\ref{LKcondition})-(\ref{gapcondition})$, then the exponential system
$\{e^{\lambda_n t}\}_{n=1}^{\infty}$
is hereditarily complete in its closed span in $L^2 (a,b)$. The result below extends the work of \cite{Zikkos2024JMAA}.

\begin{theorem}\label{hereditarycompleteness}
Consider all the assumptions of Theorem $\ref{biorthogonalsystem}$.
Suppose however that $m\le w(x)\le M$ on $A$ for some positive numbers $m$ and $M$ thus by Remark $\ref{equalr}$ we have $r_w = r_A$.
Suppose also that the set $A$ contains an interval $[a, r_A]$ for some $a\ge 0$.
Then the system $M_{\Lambda}$ is hereditarily complete in its closed span in $L^2_w (A)$.
\end{theorem}

We remark that hereditary completeness is related to the $Spectral\,\, Synthesis$ problem for linear operators
\cite{Markus1970,Baranov2013,Baranov2015,Baranov2022}.
Let $T$ be a bounded linear operator in a separable Hilbert space $\cal H$ such that $T$ has a set of eigenvectors
which is complete in $\cal H$. Then $T$ admits $Spectral\,\, Synthesis$ if for any invariant subspace $A$ of $T$,
the set of eigenvectors of $T$ contained in $A$ is complete in $A$.
In \cite{Wermer1952} J. Wermer proved that if an operator $T$ is $\bf both$ compact and normal, then $T$
admits spectral synthesis. The conclusion might not be true if one of these two conditions does not hold
(see \cite[Theorem 2]{Wermer1952} and \cite[Theorem 4.2]{Markus1970}). To this end we point out that
A. S. Markus \cite[Theorem 4.1]{Markus1970} proved that if a compact operator $T$ has a trivial kernel and non-zero simple eigenvalues with corresponding eigenvectors $\{f_n\}$, then $T$ admits spectral synthesis if and only if the family $\{f_n\}$ is hereditarily complete in $\cal H$.

Intrigued by the above, we present below a class of compact, non-normal operators on the closed span of $M_{\Lambda}$ in $L^2_w (A)$
that admit spectral synthesis.

\begin{theorem}\label{spectralsynthesis}
Consider all the assumptions of Theorem $\ref{hereditarycompleteness}$. Denote by $[\overline{\text{span}}(M_{\Lambda})]_{w, A}$
the closed span of the system $M_{\Lambda}$ in $L^2_w (A)$.
Fix a sequence $\{u_n\}_{n=1}^{\infty}$ of distinct non-zero real numbers such that
\begin{equation}\label{un}
|u_n|\le  \rho^{\lambda_n}\qquad \text{for\,\, some}\qquad 0<\rho<1.
\end{equation}
Then $T : [\overline{\text{span}}(M_{\Lambda})]_{w, A}\to[\overline{\text{span}}(M_{\Lambda})]_{w, A}$
defined as
\[
Tf(x) : = \sum_{n=1}^{\infty} \langle f , r_n\rangle_{w, A}\cdot u_n\cdot  x^{\lambda_n},
\]
is an operator which is compact, not normal, and admits Spectral Synthesis.
In particular, this is true for the operator $T_{\rho}(f):=f(\rho x)$
for any fixed $0<\rho<1$.
\end{theorem}

Our final result in this article is on $Moment$ problems (see Theorem $\ref{MomentProblem}$). We note that when
$\{\lambda_n\}_{n=1}^{\infty}$ is a sequence satisfying $(\ref{LKcondition})-(\ref{gapcondition})$ or similar conditions,
and $\{r_n\}_{n=1}^{\infty}$ is the unique biorthogonal family to the system $\{e^{-\lambda_n x}\}_{n=1}^{\infty}$ in the closed span of the
latter in $L^2 (0, T)$, then searching for solutions $f$ to moment problems such as
\[
\int_{0}^{T} f(x)\cdot e^{-\lambda_n x} \, dx=
d_{n}\qquad \forall\,\, n\in\mathbb{N},
\]
is very useful in Control Theory for Partial Differential Equations (\cite{1971FR,2011JPA,2019JPA}),
starting with the pioneering work of Fattorini and Russell \cite{1971FR} and followed by a vast amount of work done after that.

Inspired by this we derived the following result.
\begin{theorem}\label{MomentProblem}
Let $A$  be a measurable subset of the real half-line $[0,\infty)$ of positive Lebesgue measure
and let $w$ be a  real-valued non-negative integrable function defined on $A$, with $r_w$ as in $(\ref{rw})$.
Let $\Lambda=\{\lambda_n\}_{n=1}^{\infty}$ be a sequence of distinct positive real numbers that
satisfies $(\ref{LKcondition})-(\ref{gapcondition})$.
Consider a sequence $\{d_n\}_{n=1}^{\infty}$ of real numbers such that
\begin{equation}\label{momentbound}
|d_n|=O(a^{\lambda_n})\qquad \text{for\,\, some}\qquad a\in [0, r_w).
\end{equation}
Then there exists a unique function $f\in\overline{\text{span}}(M_{\Lambda})$
in $L^2_w (A)$, serving as a solution to the moment problem
\begin{equation}\label{momentproblemquestion}
\int_{A} f(x)\cdot  x^{\lambda_n}\cdot w(x)\, dx=d_n, \qquad n=1,2,\cdots
\end{equation}
\end{theorem}
\begin{remark}
Certain concepts from Non-Harmonic Fourier series such as Bessel sequences and Riesz-Fischer sequences will be utilized
to prove the above result.
\end{remark}

The rest of this article is organized as follows:
in Section 3 we revisit Theorem $B$ and prove Lemma $\ref{revisiting}$ as well as some auxiliary result.
The proofs of Theorems $\ref{Distances}$, $\ref{biorthogonalsystem}$, $\ref{hereditarycompleteness}$, and $\ref{MomentProblem}$ are given respectively in
Sections 2, 4, 5, and 6. The proof of Theorem $\ref{spectralsynthesis}$ is given in the Appendix section since it is almost identical
to the one given for \cite[Theorem 4.1]{Zikkos2024JMAA}.

\section{Proof of Theorem $\ref{Distances}$ on Distances}
\setcounter{equation}{0}

Fix some small positive $\epsilon$ and let $\delta_{\epsilon}:=r_w-\frac{\epsilon}{2}$. By the definition of $r_w$ in $(\ref{rw})$,
we claim that there exists an $\alpha_{\epsilon}>0$, which depends on $\epsilon$, so that the set
\begin{equation}\label{Bepsilon}
B_{\epsilon}:=\{x\in A\cap (\delta_{\epsilon}, \infty): \, w(x)>\alpha_{\epsilon}\}
\end{equation}
has $\bf positive$ Lebesgue measure. Let us justify this argument by writing the set
\[
A\cap (\delta_{\epsilon}, \infty)
\]
as a countable union of subsets
\[
\{x\in A\cap (\delta_{\epsilon}, \infty): \, w(x)=0\}\cup \{x\in A\cap (\delta_{\epsilon}, \infty): \, w(x)>1\}\cup
\left(\bigcup_{n=1}^{\infty} \left\{x\in A\cap (\delta_{\epsilon}, \infty): \, \frac{1}{n+1}<w(x)\le \frac{1}{n}\right\}\right).
\]
If for each positive $\alpha$ the measure of the set $\{x\in A\cap (\delta_{\epsilon}, \infty): \, w(x)>\alpha\}$ is equal to zero, then the measure of the set $A\cap (\delta_{\epsilon}, \infty)$ will be equal to the measure of the set $\{x\in A\cap (\delta_{\epsilon}, \infty): \, w(x)=0\}$. This means that $\int_{A\cap (\delta_{\epsilon}, \infty)} w(t)\, dt=0$ which contradicts the definition of $r_w$ in $(\ref{rw})$.
Hence for some $\alpha_{\epsilon}>0$  there exists a set $B_{\epsilon}$ of positive Lebesgue measure
so that $w(x)>\alpha_{\epsilon}$ on $B_{\epsilon}$.

Then one has
\[
||f||_{L^p_w (A)}\ge ||f||_{L^p_w (B_{\epsilon})}\ge \alpha^{1/p}_{\epsilon}\cdot ||f||_{L^p (B_{\epsilon})}.
\]
Since $B_{\epsilon}\subset [\delta_{\epsilon}, r_w]$
it then follows from Theorem $A$, relation $(\ref{Erdelyi})$, that for all functions $f$ in the span of $M_{\Lambda}$,
there exists a positive constant $c_{\epsilon}$ which depends on $\epsilon$ and on the measure of the set $B_{\epsilon}$,
such that
\[
||f||_{L^p (B_{\epsilon})}\ge c_{\epsilon}\cdot ||f||_{[0,\delta_{\epsilon}]}.
\]
Choose any $a\in (0, \delta_{\epsilon})$. Clearly one has $||f||_{[0,\delta_{\epsilon}]}\ge ||f||_{[a,\delta_{\epsilon}]}$.

Then, combining all of the above shows that
\begin{equation}\label{alphac}
||f||_{L^p_w (A)}\ge \alpha^{1/p}_{\epsilon}\cdot c_{\epsilon}\cdot ||f||_{[a,\delta_{\epsilon}]}.
\end{equation}
Now, let $g$ belong to the span of the system $M_{\Lambda_n}$ and let $e_n(x)=x^{\lambda_n}$, thus $e_n-g$ belongs to the span of $M_{\Lambda}$.
Then from above we get
\begin{equation}\label{june16}
||e_n-g||_{L^p_w (A)}\ge \alpha^{1/p}_{\epsilon}\cdot c_{\epsilon}\cdot ||e_n-g||_{[a,\delta_{\epsilon}]}.
\end{equation}

Next, we recall a lower bound for $||e_n-g||_{[a,\delta_{\epsilon}]}$ due to
Luxemburg and Korevaar. They proved in \cite[Theorem 7.1 and relation (1.9)]{LK} that for every positive
$\eta$ there exists a positive constant $M_{\epsilon,\eta}$, which depends on $\eta$ and on the interval $[a,\delta_{\epsilon}]$,
hence on $\epsilon$, such that
\[
\inf_{g\in \overline{\text{span}} (M_{\Lambda_n})}||e_n-g||_{[a,\delta_{\epsilon}]}\ge M_{\epsilon, \eta}\cdot (\delta_{\epsilon}-\eta)^{\lambda_n}.
\]
For the fixed $\epsilon>0$, take $\eta=\epsilon/2$ and recall that $\delta_{\epsilon}=r_w-\epsilon/2$ hence $\delta_{\epsilon}-\epsilon/2=r_w-\epsilon$.
Thus there exists a constant $M_{\epsilon}$, such that
\[
\inf_{g\in \overline{\text{span}} (M_{\Lambda_n})}||e_n-g||_{[a,\delta_{\epsilon}]}\ge M_{\epsilon}\cdot
\left(\delta_{\epsilon}-\frac{\epsilon}{2}\right)^{\lambda_n}=M_{\epsilon}\cdot (r_w-\epsilon)^{\lambda_n}.
\]
Together with $(\ref{june16})$ shows that for all $g\in \text{span}(M_{\Lambda_n})$, we have
\[
 ||e_n-g||_{L^p_w (A)}\ge
\alpha^{1/p}_{\epsilon}\cdot c_{\epsilon}\cdot M_{\epsilon}\cdot (r_w-\epsilon)^{\lambda_n}.
\]
Finally, letting $m_{\epsilon}=\alpha^{1/p}_{\epsilon}\cdot c_{\epsilon}\cdot M_{\epsilon}$ gives us the distance result
$(\ref{distancelowerbounds})$, that is
\[
\inf_{g\in \overline{\text{span}}(M_{\Lambda_n})}||e_n-g||_{L^p_w (A)}\ge
m_{\epsilon}\cdot (r_w-\epsilon)^{\lambda_n}.
\]

\section{Proof of Lemma $\ref{revisiting}$ and some auxiliary result}
\setcounter{equation}{0}

In this section we first revisit Theorem B by proving Lemma $\ref{revisiting}$.

\subsection{Proof of Lemma $\ref{revisiting}$}

As in the proof of Theorem $\ref{Distances}$,
$\bf fix$ some small positive $\epsilon$ and let $\delta_{\epsilon}:=r_w-\frac{\epsilon}{2}$. Then
there exists an $\alpha_{\epsilon}>0$, which depends on $\epsilon$, so that the set
$B_{\epsilon}$ as in $(\ref{Bepsilon})$ has positive Lebesgue measure.
Since $||f-P_j||_{L^p_{w,A}}\to 0$ as $j\to\infty$ then $||f-P_j||_{L^p (B_{\epsilon})}\to 0$ as $j\to\infty$ as well.
Thus $\{P_j\}_{j=1}^{\infty}$ is a Cauchy sequence in $L^p (B_{\epsilon})$.
It follows by Theorem $A$ relation $(\ref{Erdelyi})$, that $\{P_j\}_{j=1}^{\infty}$ is uniformly Cauchy on $[0, \delta_{\epsilon}]$.
Hence there exists a function $g_{\delta_{\epsilon}}\in C[0, \delta_{\epsilon}]$, the space of continuous functions on $[0, \delta_{\epsilon}]$,
such that $|g_{\delta_{\epsilon}}(x)- P_j(x)|\to 0$ uniformly on $[0, \delta_{\epsilon}]$ as $j\to\infty$. Thus we also have
\begin{equation}\label{gdelta}
||g_{\delta_{\epsilon}} - P_j||_{L^p (0, \delta_{\epsilon})}\to 0 \quad \text{as}\quad j\to\infty,\qquad
\text{and}\qquad ||P_j||_{L^p (0, \delta_{\epsilon})} \to ||g_{\delta_{\epsilon}}||_{L^p (0, \delta_{\epsilon})}\quad\text{as}\quad j\to\infty.
\end{equation}

From the pre-mentioned result by Luxemburg and Korevaar \cite[Theorem 7.1 and relation (1.9)]{LK} it follows that for every positive
$\eta$ there exists a positive constant $M_{\eta}$, such that
\begin{equation}\label{LKresult}
\inf_{h\in \overline{\text{span}} (M_{\Lambda_n})}||e_n-h||_{L^p([0,\delta_{\epsilon}])}\ge M_{\eta}\cdot (\delta_{\epsilon}-\eta)^{\lambda_n}.
\end{equation}
Now, $P_j(x)=\sum_{n=1}^{p(j)} a_{j,n} x^{\lambda_n}$: fix $k\in\{1, 2, \dots, p(j)\}$ and write
\[
P_j(x)=a_{j,k}\cdot \left(x^{\lambda_k}+ \sum_{n=1,\, n\not=k}^{p(j)} \frac{a_{j,n}}{a_{j,k}} x^{\lambda_n} \right).
\]
Then let $H_{j,k}(x) : = \sum_{n=1, \, n\not=k}^{p(j)} \frac{a_{j,n}}{a_{j, k}} x^{\lambda_n}$
thus $H_{j, k}$ belongs to the space $M_{\Lambda_k}$ $(\ref{exceptone})$. Combined with $(\ref{LKresult})$, we get
\[
||P_j||_{L^p (0, \delta_{\epsilon})} = |a_{j,k}|\cdot ||e_k+H_{j,k}||_{L^p (0, \delta_{\epsilon})}\ge
|a_{j,k}|\cdot M_{\eta}\cdot (\delta_{\epsilon}-\eta)^{\lambda_k}.
\]
Letting  $m_{\eta} = 1/M_{\eta}$ gives
\begin{equation}\label{Cauchy2}
|a_{j,k}|\le m_{\eta}\cdot ||P_j||_{L^p (0, \delta_{\epsilon})}\cdot (\delta_{\epsilon}-\eta)^{-\lambda_n}.
\end{equation}
Similarly one gets
\begin{equation}\label{Cauchy1}
|a_{i,k}-a_{j,k}|\le m_{\eta}\cdot ||P_i-P_j||_{L^p (0, \delta_{\epsilon})}\cdot (\delta_{\epsilon}-\eta)^{-\lambda_n}.
\end{equation}
Now, it follows from $(\ref{gdelta})$ that $\{P_j\}_{j=1}^{\infty}$ is Cauchy in $L^p (0, \delta_{\epsilon})$.
Thus, from $(\ref{Cauchy1})$ we see that for every fixed $k\in\mathbb{N}$,
$\{a_{j,k}\}_{j=1}^{\infty}$ is a Cauchy sequence of real numbers, hence
\begin{equation}\label{limitlimit}
a_{j,k}\to a_k\qquad \text{for some}\quad a_k\in\mathbb{R}.
\end{equation}
Moreover, from $(\ref{gdelta})$ and $(\ref{Cauchy2})$ we get
\begin{equation}\label{ak}
|a_{k}|\le m_{\eta}\cdot ||g_{\delta_{\epsilon}}||_{L^p (0, \delta_{\epsilon})}\cdot (\delta_{\epsilon}-\eta)^{-\lambda_k}.
\end{equation}
If we fix the index $i$ in $(\ref{Cauchy1})$ and let the index $j\to\infty$, we then get from $(\ref{gdelta})$ and $(\ref{Cauchy1})$
\begin{equation}\label{ajk}
|a_{i,k}-a_{k}|\le m_{\eta}\cdot ||P_i-g_{\delta_{\epsilon}}||_{L^p (0, \delta_{\epsilon})}\cdot (\delta_{\epsilon}-\eta)^{-\lambda_k}.
\end{equation}

We will need the estimates $(\ref{ak})-(\ref{ajk})$ below. First of all, it follows from $(\ref{ak})$ that
\begin{equation}\label{an}
F_{\delta_{\epsilon}}(z) : = \sum_{n=1}^{\infty} a_n z^{\lambda_n}
\end{equation}
defines an analytic function in the slit disk
\[
D_{\delta_{\epsilon}}:=\{z\in\mathbb{C}\setminus (-\infty, 0]: \,\, |z|<\delta_{\epsilon}\}.
\]
converging uniformly on compact subsets of $D_{\delta_{\epsilon}}$.

We will now show below that
\begin{equation}\label{g=F}
F_{\delta_{\epsilon}}(x) = f(x)\qquad\text{almost everywhere}\quad \text{on}\quad A\cap [0, \delta_{\epsilon}).
\end{equation}

$\bf Fix$ any $\rho\in (0, \delta_{\epsilon})$ and choose $\eta>0$ such that $\delta_{\epsilon}-\eta>\rho$, thus $\rho/(\delta_{\epsilon}-\eta)<1$. Then the series
\[\sum_{n=1}^{\infty} (\delta_{\epsilon}-\eta)^{-\lambda_n}\cdot x^{\lambda_n}
=\sum_{n=1}^{\infty}\left(\frac{x}{\delta_{\epsilon}-\eta}\right)^{\lambda_n}
\]
converges uniformly on the interval $[0, \rho]$. Therefore,
\begin{equation}\label{unif1}
\sum_{n=p(j)+1}^{\infty} (\delta_{\epsilon}-\eta)^{-\lambda_n}\cdot  x^{\lambda_n}\to 0\qquad \text{uniformly on}\quad [0, \rho]\quad
\text{as}\quad  j\to\infty
\end{equation}
and there exists some $M>0$ so that
\begin{equation}\label{unif2}
\sum_{n=1}^{p(j)} (\delta_{\epsilon}-\eta)^{-\lambda_n}\cdot x^{\lambda_n}<M\qquad \text{for all}\quad x\in [0,\rho].
\end{equation}

Then for all $x\in [0,\rho]$ write
\begin{eqnarray*}
\left|F_{\delta_{\epsilon}}(x)-P_j(x)\right| & = & \left|\sum_{n=1}^{\infty} a_n x^{\lambda_n}-\sum_{n=1}^{p(j)} a_{j,n} x^{\lambda_n}\right|\\
& = & \left|\sum_{n=p(j)+1}^{\infty} a_n x^{\lambda_n}-\sum_{n=1}^{p(j)} (a_{j,n}-a_n) x^{\lambda_n}\right|\\
& \le & \sum_{n=p(j)+1}^{\infty} |a_n| x^{\lambda_n} + \sum_{n=1}^{p(j)} |a_{j,n}-a_n| x^{\lambda_n}.
\end{eqnarray*}
If we replace the upper bounds from $(\ref{ak})-(\ref{ajk})$ we get for all $x\in [0,\rho]$
\[
\left|F_{\delta_{\epsilon}}(x)-P_j(x)\right|\le
m_{\eta}\cdot ||g_{\delta_{\epsilon}}||_{L^p (0, \delta_{\epsilon})}\cdot
\sum_{n=p(j)+1}^{\infty} (\delta_{\epsilon}-\eta)^{-\lambda_n}\cdot  x^{\lambda_n} +
 m_{\eta}\cdot ||P_j-g_{\delta_{\epsilon}}||_{L^p (0, \delta_{\epsilon})}\cdot \sum_{n=1}^{p(j)} (\delta_{\epsilon}-\eta)^{-\lambda_n}\cdot x^{\lambda_n}.
\]
It then follows from $(\ref{gdelta})$, $(\ref{unif1})$ and $(\ref{unif2})$ that

\[
\left|F_{\delta_{\epsilon}}(x)-P_j(x)\right|\to 0\qquad\text{uniformly on}\quad [0, \rho]\quad\text{as}\quad j\to\infty.
\]
Thus for every positive $\tau>0$,
there exists some $j_{\tau}\in\mathbb{N}$ such that for all $j\ge j_{\tau}$ one has $|F_{\delta_{\epsilon}}(x)- P_j(x)|<\tau$
for all $x\in [0, \rho]$. Hence
\[
\int_{A\cap [0, \rho]} |F_{\delta_{\epsilon}}(x)- P_j(x)|^p\cdot w(x)\, dx\le \tau^p\cdot \int_A w(x)\, dx,
\qquad \text{for all}\quad j\ge j_{\tau}.
\]
Since $\int_A w(x)\, dx<\infty$ we get
\begin{equation}\label{aaa}
||F_{\delta_{\epsilon}}-P_j||_{L^p_{w, A\cap [0,\rho]}}\to 0\quad \text{as}\quad j\to\infty.
\end{equation}
But $||f-P_j||_{L^p_{w, A\cap [0,\rho]}}\to 0$ as $j\to\infty$ since $||f-P_j||_{L^p_{w, A}}\to 0$ as $j\to\infty$.
Combined together shows that
\[
||f-F_{\delta_{\epsilon}}||_{L^p_{w, A\cap [0,\rho]}}=0.
\]
Due to $(\ref{measurezero})$ we see that
\[
F_{\delta_{\epsilon}}(x)=f(x)\qquad\text{almost everywhere}\quad \text{on}\quad  A\cap [0, \rho].
\]
However, this is true for any fixed $\rho \in (0, \delta_{\epsilon})$. Hence we conclude that
\[
f(x)=F_{\delta_{\epsilon}}(x)\qquad \text{almost everywhere}\quad \text{on}\quad A\cap [0, \delta_{\epsilon}).
\]
In other words,
\begin{equation}\label{equal1l}
f(x)=\sum_{n=1}^{\infty} a_n x^{\lambda_n}\qquad \text{almost everywhere}\quad \text{on}\quad A\cap [0, \delta_{\epsilon}).
\end{equation}

Using the same rational, and repeating the above arguments, if we choose $\delta_1$ such that $r_w>\delta_1>\delta_{\epsilon}$, then
there is a function $F_{\delta_1}(x)$,
\begin{equation}\label{bn}
F_{\delta_1}(x)=\sum_{n=1}^{\infty} b_n x^{\lambda_n},\qquad b_n\in\mathbb{R},
\end{equation}
defining an analytic function in the slit disk
\[
D_{\delta_1}:=\{z\in\mathbb{C}\setminus (-\infty, 0]: \,\, |z|<\delta_1\},
\]
converging uniformly on compact subsets of $D_{\delta_1}$, and such that
\[
f(x)=F_{\delta_1}(x)\qquad \text{almost everywhere}\quad \text{on}\quad A\cap [0, \delta_1).
\]
Observe however, that due to the $\bf Uniqueness$ of limits, the limit in $(\ref{limitlimit})$ is unique.
Thus the coefficients $b_n$ in $(\ref{bn})$ are equal to the respective coefficients $a_n$ in $(\ref{an})$.
Hence, $F_{\delta_1}(x)$ is identical to $F_{\delta_{\epsilon}}(x)$. This implies that
$(\ref{analyticextension})$ is valid for all $z\in D_{r,w}$ (in the almost everywhere sense on $A$).

\subsection{An auxiliary result}

\begin{lemma}\label{span}
Suppose that $f\in\overline{\text{span}} (M_{\Lambda})$  in $L^p_w (A)$, $p\ge 1$,
thus by Theorem $B$ and Lemma $\ref{revisiting}$  we have $f(x)=\sum_{n=1}^{\infty} a_n x^{\lambda_n}$ almost everywhere on $A\cap [0, r_w)$.
Then, for any fixed element $x^{\lambda_n}$, the function
\[
f(x)-a_n x^{\lambda_n}
\]
belongs to the closed span of the system $M_{\Lambda_n}$ $(\ref{exceptone})$ in $L^p_w (A)$.
\end{lemma}

\begin{proof}
Since $f\in\overline{\text{span}} (M_{\Lambda})$  in $L^p_w (A)$, then a sequence $\{P_j(x)\}_{j=1}^{\infty}$ in $\text{span} (M_{\Lambda})$
exists, where $P_j(x)=\sum_{n=1}^{p(j)} a_{j,n}x^{\lambda_n}$,
such that $||f-P_j||_{L^p_w (A)}\to 0$ as $j\to\infty$. By Theorem $B$ and Lemma $\ref{revisiting}$, $f$ is extended analytically
in the slit disk $D_{r_w}$ $(\ref{slit})$ such that $f(x)=\sum_{n=1}^{\infty} a_n x^{\lambda_n}$
for almost all $x\in A\cap [0,r_w]$ and $a_n=\lim_{j\to\infty} a_{j,n}$ (see relation $(\ref{important})$).
Hence, for every fixed $n\in\mathbb{N}$ we have
\[
(a_{j,n}-a_{n}) x^{\lambda_n}\to 0\quad as\quad j\to \infty,\qquad
\text{uniformly on}\quad [0,r_w].
\]
That is, for any $\epsilon>0$ there exists some $j_{\epsilon}\in\mathbb{N}$ so that for all $j>j_{\epsilon}$ and for all
$x\in [0,r_w]$ one has $\left|(a_{j,n}-a_{n}) x^{\lambda_n }\right|^p<\epsilon$. Therefore for all $j>j_{\epsilon}$ we have
\[
\int_{A\cap [0,r_w]}\left|(a_{j,n}-a_{n}) x^{\lambda_n }\right|^p\cdot w(x)\, dx<\epsilon \int_{A\cap [0,r_w]} w(x)\, dx.
\]
Thus
\[
\lim_{j\to\infty}\int_{A\cap [0,r_w]}\left|(a_{j,n}-a_{n}) x^{\lambda_n }\right|^p\cdot w(x)\, dx=0.
\]
Clearly one also has
\[
\lim_{j\to\infty}\int_{A}\left|(a_{j,n}-a_{n}) x^{\lambda_n }\right|^p\cdot w(x)\, dx=0.
\]
But $||f-P_j||_{L^p_w (A)}\to 0$ as $j\to\infty$. Combined together and applying the Minkowski inequality shows that
\[
\lim_{j\to\infty}\int_{A}\left|\left[ f(x)-P_j(x)\right] -
\left[a_n x^{\lambda_n} - a_{j,n} x^{\lambda_n}\right]\right|^p\cdot w(x)\, dx=0.
\]
We rewrite this as
\begin{equation}\label{final}
\lim_{j\to\infty}\int_{A}\left|\left[ f(x)-a_n x^{\lambda_n}\right] -
\left[P_j(x)- a_{j,n} x^{\lambda_n}\right]\right|^p\cdot w(x)\, dx=0.
\end{equation}

Clearly the function $P_j(x)- a_{j,n} x^{\lambda_n}$
belongs to the span of the system $M_{\Lambda_n}$ $(\ref{exceptone})$. We then conclude from $(\ref{final})$ that the function
$ f(x)-a_n x^{\lambda_n}$ belongs to the closed span of the system $M_{\Lambda_n}$ in $L^p_w(A)$.

\end{proof}

\section{Proof of Theorem $\ref{biorthogonalsystem}$}
\setcounter{equation}{0}

In this section we prove Theorem $\ref{biorthogonalsystem}$ on the existence of a biorthogonal family $r_{\Lambda}$ to the system $M_{\Lambda}$
and the Fourier-type series representations $(\ref{representationf})$ for functions in the closed span of $M_{\Lambda}$ in $L^2 _w (A)$.
We also show that $M_{\Lambda}$ is a Markushevich basis for its closed span in $L^2_w(A)$

\subsection{Constructing the Biorthogonal family and deriving the upper bound $(\ref{rnkbound})$}

In $(\ref{distancelowerbounds})$ we derived a positive lower bound for $D_{A,w,2,n}$ which is the distance
between a function $x^{\lambda_n}$ and the closed span of the system $M_{\Lambda_n}$ $(\ref{exceptone})$ in $L^2_w(A)$.
Since $L^2_w(A)$ is a separable Hilbert space,
it then follows from the $Closest\,\, Point\,\, Theorem$
that there exists a unique element in $\overline{\text{span}}(M_{\Lambda_{n}})$ in $L^2_w(A)$,
that we denote by $\phi_{n}$, so that
\[
||e_{n}-\phi_{n}||_{L^2_w (A)}=
\inf_{g\in \overline{\text{span}}(M_{\Lambda_{n}})}||e_{n,k}-g||_{L^2_w (A)}=D_{A,w,2,n}.
\]
The function $e_{n}-\phi_{n}$ is orthogonal to all the elements of the closed span of $M_{\Lambda_{n}}$
in $L^2_w (A)$, hence to $\phi_{n}$ itself.
Therefore
\[
\langle e_{n}-\phi_{n}, e_{n}-\phi_{n} \rangle_{w,A}=\langle e_{n}-\phi_{n}, e_{n}\rangle_{w,A}.
\]
Hence
\[
(D_{A,w,2,n})^2=\langle e_{n}-\phi_{n}, e_{n}\rangle_{w,A}.
\]
Next, we define
\[
r_{n}(x):=\frac{e_{n}(x)-\phi_{n}(x)}{(D_{A,w,2,n}) ^2}.
\]
It then follows that $\langle r_{n}, e_{n}\rangle_{w,A}=1$ and $r_{n}$ is orthogonal to all the
elements of the system $M_{\Lambda_{n}}$.
Thus $\{r_{n}:\,\, n\in\mathbb{N}\}$ is biorthogonal  to the system $M_{\Lambda}$ in $L^2_w (A)$.
Since $\phi_{n}\in\overline{\text{span}}(M_{\Lambda_{n}})$ in $L^2_w (A)$ then
$r_{n}\in\overline{\text{span}}(M_{\Lambda})$ in $L^2_w (A)$.

\begin{remark}
Clearly we have $||r_{n}||_{L^2_w (A)}=\frac{1}{D_{A,w,2,n}}$ thus $(\ref{rnkbound})$ follows from $(\ref{distancelowerbounds})$.
\end{remark}

Next we show that $\{r_{n}\}$ is the unique biorthogonal sequence to the system $M_{\Lambda}$,
which belongs to its closed span in $L^2_w (A)$. Indeed, if there is another such biorthogonal sequence,
call it $\{q_{n}\}$, then for all $n\in\mathbb{N}$ we have
\[
\langle r_{n}-q_{n}, e_{m}\rangle_{w,A}=0, \qquad \forall\,\, m\in\mathbb{N}.
\]
But this in turn implies that $r_{n}-q_{n}=0$ almost everywhere on $A$ since the system $M_{\Lambda}$
is complete in its closed span in $L^2_w (A)$.

\subsection{The Fourier-type series representations $(\ref{representationf})$}

By Theorem $B$, a function $f\in\overline{\text{span}} (M_{\Lambda})$ in $L^2_w (A)$
extends analytically in the slit disk $D_{r_w}$ $(\ref{slit})$ and
$f(x)=\sum_{n=1}^{\infty} a_n x^{\lambda_n}$ almost everywhere on $A\cap [0, r_w)$.
To obtain $(\ref{representationf})$ we will show that
\begin{equation}\label{coefcnk}
\langle f, r_{n} \rangle_{w,A} = a_n.
\end{equation}

Due to biorthogonality we get
\begin{eqnarray}
\langle f, r_{n} \rangle_{w,A} & = & \int_{A} f(x)\cdot r_n(x)\cdot w(x)\, dx \nonumber\\
& = & \int_{A} r_{n}(x) \cdot a_n x^{\lambda_n}\cdot w(x) \, dx +
\int_{A} r_{n}(x) \cdot \left[f(x)-a_n x^{\lambda_n}\right]\cdot w(x) \, dx\nonumber\\
& = & a_n +
\int_{A} r_{n}(x)\cdot \left[f(x)-a_n x^{\lambda_n}\right]\cdot w(x) \, dx.\label{cnk}
\end{eqnarray}

Now, it follows from Lemma $\ref{span}$ that the function $f_n(x):=f(x)-a_n x^{\lambda_n}$ belongs
to the closed span of the system $M_{\Lambda_n}$ $(\ref{exceptone})$
in $L^2_w(A)$. Hence, for every $\epsilon>0$, there is a function $g_{\epsilon}$ in the span of $M_{\Lambda_n}$
so that $||f_n-g_{\epsilon}||_{L^2_w(A)}<\epsilon$. Due to the biorthogonality we have
\[
\int_{A} r_n(x)\cdot g_{\epsilon}(x)\cdot w(x)\, dx=0.
\]
Combining with the Cauchy-Schwarz inequality we get
\begin{eqnarray*}
\left|\int_{A} r_n(x)\cdot f_n(x)\cdot w(x)\, dx\right|
& = & \left|\int_{A} r_n(x)\cdot \left[f_n(x)-g_{\epsilon}(x)\right]\cdot w(x)\, dx\right|\\
& = & \left|\int_{A} \left(r_n(x)\cdot \sqrt{w(x)}\right)\cdot \left(\left[f_n (x)-g_{\epsilon}(x)\right]\cdot \sqrt{w(x)}\right)\, dx\right|\\
& \le & \epsilon\cdot ||r_n||_{L^2_w(A)}.
\end{eqnarray*}
The arbitrary choice of $\epsilon$ implies that $\int_{A} r_n(x)\cdot f_n(x)\cdot w(x)\, dx=0$, that is
\[
\int_{A} r_n(x)\cdot \left[f(x)-a_n x^{\lambda_n}\right]\cdot w(x)\, dx=0.
\]
Replacing in $(\ref{cnk})$ shows that $(\ref{coefcnk})$ holds.

\subsection{The associated function $f^*$ of $f\in L^2_w (A)$}

Clearly $f(x)$ can be written uniquely as
\[
f(x)=g(x)+h(x)
\]
where

(I) $g$ belongs to the closed span of the system $M_{\Lambda}$ in $L^2_w (A)$, and

(II) $h$ belongs to the orthogonal complement of the closed span of the system $M_{\Lambda}$ in $L^2_w (A)$,
thus $\langle h, r_n\rangle_{w, A}=0$ for all $r_n\in r_{\Lambda}$.

Since $h=f-g$, we then have
\[
\langle g, r_n\rangle_{w, A}=\langle f, r_n\rangle_{w, A}\qquad \text{for all}\quad n\in\mathbb{N}.
\]
It follows from $(\ref{representationf})$ that
\[
g(x)=\sum_{n=1}^{\infty}\langle g, r_n \rangle_{w, A} \cdot x^{\lambda_n}\qquad \text{almost everywhere on}\quad A.
\]
Combining the above shows that
\[
f^*(z)=\sum_{n=1}^{\infty}\langle f, r_n \rangle_{w, A}\cdot  z^{\lambda_n}
\]
belongs to the closed span of $M_{\Lambda}$ in $L^2_w (A)$.

\subsection{The system $M_{\Lambda}$ is a Markushevich basis for its closed span in $L^2_w(A)$}

We have to show that
\[
\overline{\text{span}}(r_{\Lambda})=\overline{\text{span}}(M_{\Lambda})\quad \text{in}\quad L^2_w(A).
\]

Let us denote $\overline{\text{span}}(M_{\Lambda})$ in $L^2_w (A)$  by  $[\overline{\text{span}}(M_{\Lambda})]_{w, A}$  and let
$[\overline{\text{span}}(r_{\Lambda})]_{w, A}$ be the closed span of $r_{\Lambda}$ in $L^2_w (A)$.
Obviously $[\overline{\text{span}}(r_{\Lambda})]_{w, A}$ is a subspace of $[\overline{\text{span}}(M_{\Lambda})]_{w, A}$.
Let $[\overline{\text{span}}(r_{\Lambda})]_{w, A}^{\perp}$ be the orthogonal complement  of $[\overline{\text{span}}(r_{\Lambda})]_{w, A}$ in $[\overline{\text{span}}(M_{\Lambda})]_{w, A}$, that is
\[
[\overline{\text{span}}(r_{\Lambda})]_{w, A}^{\perp}=\{f\in [\overline{\text{span}}(M_{\Lambda})]_{w, A}:\,\, \langle f, g \rangle_{w,A}
=0 \quad for\,\, all\,\, g\in [\overline{\text{span}}(r_{\Lambda})]_{w, A}\}.
\]
Now, if $f\in [\overline{\text{span}}(M_{\Lambda})]_{w, A}$ then as shown earlier
\[
f(x)=\sum_{n=1}^{\infty}\langle f, r_{n} \rangle_{w,A} x^{\lambda_n}, \quad \text{almost everywhere on}\,\, A\cap [0,r_w].
\]
But if $f\in [\overline{\text{span}}(r_{\Lambda})]_{w, A}^{\perp}\subset [\overline{\text{span}}(M_{\Lambda})]_{w, A}$
then $\langle f, r_{n} \rangle_{w,A} =0$ for all $r_{n}\in r_{\Lambda}$.
We conclude that $f=0$ almost everywhere on $A\cap [0,r_w]$, hence $[\overline{\text{span}}(r_{\Lambda})]_{w, A}^{\perp}$
contains just the zero function.
Therefore $[\overline{\text{span}}(r_{\Lambda})]_{w, A}=[\overline{\text{span}}(M_{\Lambda})]_{w, A}$.

This completes the proof of Theorem $\ref{biorthogonalsystem}$.

\section{Proof of Theorem $\ref{hereditarycompleteness}$ on Hereditary completeness}
\setcounter{equation}{0}

Let us write the set $\mathbb{N}$ as an arbitrary disjoint union of two sets $N_1$ and $N_2$.
In order to obtain the hereditary completeness of $M_{\Lambda}$ in its closed span in $L^2 (A)$,
we must show that the closed span of the mixed system
\[
M_{1,2}:=\{e_{n}:\,\, n\in N_1\}\cup\{r_{n}:\,\, n\in N_2\},
\]
in $L^2_w (A)$ is equal to $\overline{\text{span}}(M_{\Lambda})$ in $L^2_w (A)$, which it was denoted earlier by
$[\overline{\text{span}}(M_{\Lambda})]_{w, A}$.
Denote by $W_{\Lambda_{1,2}}$ the closed span of $M_{1,2}$ in $L^2_w (A)$.
Obviously $W_{\Lambda_{1,2}}$ is a subspace of $[\overline{\text{span}}(M_{\Lambda})]_{w, A}$.
Let $W^{\perp}_{\Lambda_{1,2}}$ be the orthogonal complement of $W_{\Lambda_{1,2}}$ in $[\overline{\text{span}}(M_{\Lambda})]_{w, A}$,
that is
\[
W^{\perp}_{\Lambda_{1,2}}=\{f\in [\overline{\text{span}}(M_{\Lambda})]_{w, A}:\,\, \langle f, g \rangle_{w, A} =
0\quad \text{for\,\, all}\,\, g\in W_{\Lambda_{1,2}}\}.
\]
Now, if $f\in W^{\perp}_{\Lambda_{1,2}}\subset [\overline{\text{span}}(M_{\Lambda})]_{w, A}$, then by $(\ref{representationf})$ we have
\[
f(z)=\sum_{n=1}^{\infty} \langle f, r_{n} \rangle_{w, A} \cdot z^{\lambda_n},\quad z\in D_{r_w}
\]
with the series converging uniformly on the slit disk $D_{r_w}$ $(\ref{slit})$,
and of course $\langle f, r_{n} \rangle_{w, A} =0$ for all $n\in N_2$.
Thus,
\begin{equation}\label{converse}
f(x)=\sum_{n\in N_1}\langle f, r_{n} \rangle_{w, A} \cdot x^{\lambda_n},
\quad \text{almost everywhere on}\,\, [0,r_w).
\end{equation}

\begin{remark}
Since $w$ satisfies $m\le w(x)\le M$ on $A$ for some positive numbers $m, \, M$, then by Remark $\ref{equalr}$ we have $r_A=r_w$.
\end{remark}

Since $f\in L^2_w (A)$ and by assumption for some $a\in [0, r_A)$ the interval $[a,r_A]$ is a subset of $A$, then $f\in L^2_w ([a,r_A])$.
Then the inequality $m\le w(x)$ implies that $f\in L^2 (a,r_A)$. Moreover, one has $f\in L^2 (0,r_A)$ as well
since $f$ is continuous on $[0,r_A)$ due to the analytic extension of $f$ to the disk $D_{r_A}$.

To this end we point out that there is a $\bf converse$ result (see \cite[Corollary 6.2.4]{Gurariy}, \cite[Theorem 8.2]{LK}, \cite{Korevaar1947})
to the ``Clarkson-Erd\H{o}s-Schwartz Phenomenon'' which reads as follows.

``If $g\in L^2 (0,1)$ and $g(x)=\sum_{n=1}^{\infty} a_x \cdot x^{\lambda_n}$ on $(0,1)$, then
$g$ belongs to the closed span of $\{x^{\lambda_n}\}_{n=1}^{\infty}$ in $L^2 (0,1)$''.

Therefore, since $f$ as in $(\ref{converse})$ belongs to $L^2 (0,r_A)$, it follows that
$f$ belongs to the closed span of $\{e_n\}_{n\in N_1}$
in the space $L^2 (0,r_A)$. Clearly now $f$ belongs to the closed span of $\{e_n\}_{n\in N_1}$
in $L^2 (A)$ as well. Finally, from the inequality $w(x)\le M$ on $A$ we deduce that $f$ belongs to the closed span of $\{e_n\}_{n\in N_1}$
in $L^2_w (A)$ also.

Hence, for every $\epsilon>0$ there is a function $g_{\epsilon}$ in $\text{span}(\{e_n\}_{n\in N_1}$ so that
$||f-g_{\epsilon}||_{L^2_w(A)}<\epsilon$. Write
\[
\langle f, f \rangle_{w, A} = \langle f, f-g_{\epsilon} \rangle_{w, A} + \langle f, g_{\epsilon} \rangle_{w, A}.
\]
Since $f\in W^{\perp}_{\Lambda_{1,2}}$ then $\langle f, e_{n} \rangle_{w, A} =0$ for all $n\in N_1$, thus
$\langle f, g_{\epsilon} \rangle_{w, A}=0$. Therefore,

\[
||f||^2_{L^2_w (A)}=\langle f, f \rangle_{w, A} = \langle f, f-g_{\epsilon}\rangle_{w, A} \le ||f||_{L^2_{w, A}}
\cdot ||f-g_{\epsilon}||_{L^2_{w, A}}
\le ||f||_{L^2_{w, A}}\cdot \epsilon.
\]
Hence
\[
||f||_{L^2_w (A)}\le \epsilon.
\]
This holds for every $\epsilon>0$ thus $f(x)=0$ almost everywhere on $A$, hence $W^{\perp}_{\Lambda_{1,2}}=\{\mathbf{0}\}$.
Thus $W_{\Lambda_{1,2}}=[\overline{\text{span}}(M_{\Lambda})]_{w, A}$, meaning that the system $M_{\Lambda}$
is hereditarily complete in the space $L^2_w (A)$.

The proof of Theorem $\ref{hereditarycompleteness}$ is now complete.

\section{Proof of Theorem $\ref{MomentProblem}$ on the Moment problem}
\setcounter{equation}{0}

In order to prove Theorem $\ref{MomentProblem}$, we first introduce some concepts from Non-Harmonic Fourier Series
such as Bessel sequences and Riesz-Fischer sequences.

\subsection{Bessel sequences and Riesz-Fischer sequences}

Let $\cal{H}$ be a separable Hilbert space endowed with an inner product $\langle \cdot\, , \cdot \rangle$, and consider a sequence $\{f_n\}_{n=1}^{\infty}\subset \cal{H}$. We say that (see \cite[p. 128 Definition]{Young}):

(i) $\{f_n\}_{n=1}^{\infty}$ is a $\bf Bessel$ sequence if there exists a constant $B>0$ such that
$\sum_{n=1}^{\infty}|\langle f,f_n\rangle|^2<B||f||^2$ for all $f\in \cal{H}$.

(ii) $\{f_n\}_{n=1}^{\infty}$ is a $\bf{Riesz-Fischer}$ sequence
if the moment problem $\langle f, f_n\rangle =c_n$ has at least one solution in $\cal{H}$
for every sequence $\{c_n\}$ in the space $l^2(\mathbb{N})$.

The following result stated by Casazza et al.
is an interesting connection between Bessel and Riesz-Fischer sequences.

\begin{propa}\label{casazza}\cite[Proposition 2.3, (ii)]{Casazza}

The Riesz-Fischer sequences in $\cal{H}$ are precisely the families for which a biorthogonal Bessel sequence exists.
In other words

(a) Suppose that two sequences $\{f_n\}_{n=1}^{\infty}$ and $\{g_n\}_{n=1}^{\infty}$ in $\cal{H}$
are biorthogonal. Suppose also that $\{f_n\}_{n=1}^{\infty}$ is a Bessel sequence.
Then $\{g_n\}_{n=1}^{\infty}$ is a Riesz-Fischer sequence.

(b) If $\{f_n\}_{n=1}^{\infty}$ in $\cal{H}$ is a Riesz-Fischer sequence,
then there exists a biorthogonal Bessel sequence $\{g_n\}_{n=1}^{\infty}$.

\end{propa}

And now a sufficient condition so that two biorthogonal families in $\cal{H}$ are Bessel and Riesz-Fischer sequences.
The result follows from \cite[Proposition 3.5.4]{Christensen} and Proposition $\bf A$.
\begin{lemma}\label{besselriesz}
Consider two biorthogonal sequences $\{u_n\}_{n=1}^{\infty}$ and $\{v_n\}_{n=1}^{\infty}$ in $\cal{H}$ and
suppose there is some $M>0$ so that
\[
\sum_{m=1}^{\infty} |\langle v_n, v_m\rangle  |<M\qquad \text{for\,\, all}\quad n=1,2,3,\dots.
\]
Then $\{v_n\}_{n=1}^{\infty}$ is a Bessel sequence in $\cal{H}$ and
$\{u_n\}_{n=1}^{\infty}$ is a Riesz-Fischer sequence in $\cal{H}$.
\end{lemma}

\subsection{Proof of Theorem $\ref{MomentProblem}$}

As before, let $[\overline{\text{span}}(M_{\Lambda})]_{w,A}$ be the closed  span of $M_{\Lambda}$ in $L^2_w (A)$.
Let $\{d_{n}:\,\, n\in\mathbb{N}\}$ be the sequence of non-zero real numbers that satisfies
$(\ref{momentbound})$. Then, for every $n\in\mathbb{N}$ define
\[
U_{n}(t):=\lambda_n d_{n}r_{n}(t)\qquad\text{and}\qquad V_{n}(t):=\frac{x^{\lambda_n}}{\lambda_n d_{n}}.
\]
Now, it easily follows that the sets
\[
\{U_{n}:\,\, n\in\mathbb{N}\}\quad\text{and}\quad\{V_{n}:\,\, n\in\mathbb{N}\}
\]
are biorthogonal in $[\overline{\text{span}}(M_{\Lambda})]_{w,A}$.\\

We will show below that $\{U_{n}\}_{n=1}^{\infty}$ and $\{V_{n}\}_{n=1}^{\infty}$ are Bessel and Riesz-Fischer sequences respectively in $[\overline{\text{span}}(M_{\Lambda})]_{w,A}$.
First, recall that for every $\epsilon>0$ there is $m_{\epsilon}>0$ so that $||r_{n}||_{L^2_w(A)}\le m_{\epsilon}(r_w-\epsilon)^{-\lambda_n}$.
Since $|d_n|=O(a^{\lambda_n})$ for some $a \in [0, r_w)$, then there is some $N>0$ so that $|d_n|\le N a^{\lambda_n}$ for all $n\in \mathbb{N}$.
Choose then
\[
\epsilon=\frac{r_w-a}{2}\qquad \text{hence}\qquad r_w-\epsilon=\frac{r_w+a}{2}\qquad\text{thus}\qquad
||r_{n}||_{L^2_w(A)}\le m_{\epsilon}\left(\frac{2}{r_w+a}\right)^{\lambda_n}.
\]
Since $r_w>a$, we can choose some positive $\gamma$ so that
\[
1<\gamma<\frac{(r_w+a)}{2a},\qquad \text{hence}\qquad \frac{2\gamma a}{r_w +a}<1.
\]
Since $\gamma>1$ there exists some $M>0$ so that $\lambda_n\le M\gamma^{\lambda_n}$ for all $n\in\mathbb{N}$.

Combining all of the above, shows that there is some positive number $\tau$ so that
\[
||U_{n}||_{L^2_w (A)}\le \lambda_n\cdot |d_n|\cdot ||r_{n}||_{L^2_w(A)}\le \tau\cdot \gamma^{\lambda_n}\cdot a^{\lambda_n}\cdot \left(\frac{2}{r_w+a}\right)^{\lambda_n}
\]
thus
\[
||U_{n}||_{L^2_w (A)}\le \tau\cdot \left(\frac{2a\gamma}{r_w+a}\right)^{\lambda_n}.
\]
By the Cauchy-Schwarz inequality we get
\begin{equation}\label{operatorbound}
|\langle U_{n},U_{m}\rangle_{w,A}|\le \tau^2\cdot \left(\frac{2a\gamma}{r_w+a}\right)^{\lambda_n}\cdot \left(\frac{2a\gamma}{r_w+a}\right)^{\lambda_m}.
\end{equation}
Thus
\[
\sum_{n=1}^{\infty}\sum_{m=1}^{\infty} |\langle U_{n},U_{m}\rangle_{w,A}|<\tau^2\cdot\sum_{n=1}^{\infty}\sum_{m=1}^{\infty}
 \left(\frac{2a\gamma}{r_w+a}\right)^{\lambda_n}\cdot \left(\frac{2a\gamma}{r_w+a}\right)^{\lambda_m}<\infty
\]
with convergence justified by the fact that the fraction $\frac{2a\gamma}{r_w+a}<1$.

It then follows from Lemma $\ref{besselriesz}$ that $\{U_{n}\}_{n=1}^{\infty}$  is a Bessel sequence in $[\overline{\text{span}}(M_{\Lambda})]_{w,A}$ and
its biorthogonal sequence $\{V_{n}\}_{n=1}^{\infty}$ is a Riesz-Fischer sequence in $[\overline{\text{span}}(M_{\Lambda})]_{w,A}$.\\

Therefore, the moment problem
\[
\int_{A}f(x)\cdot V_{n}(x)\cdot w(x)\, dx=c_{n} \qquad \forall\,\, n\in\mathbb{N},
\]
has a solution $f$ in $[\overline{\text{span}}(M_{\Lambda})]_{w,A}$ whenever $\sum_{n=1}^{\infty}|c_{n}|^2<\infty$.
Since $\sum_{n=1}^{\infty}1/\lambda_n<\infty$, we can take $c_{n}=1/\lambda_n$ for all $n\in\mathbb{N}$.
Hence, recalling the definition of $V_{n}$, there is some function $f\in [\overline{\text{span}}(M_{\Lambda})]_{w,A}$ so that
\[
\int_{A}f(x)\cdot \left( \frac{x^{\lambda_n}}{d_{n}\lambda_n}\right)\cdot w(x) \, dx=\frac{1}{\lambda_n}\qquad \forall\,\, n\in\mathbb{N}.
\]
Thus
\[
\int_{A}f(x)\cdot  x^{\lambda_n}\cdot w(x) \, dx=d_{n}\qquad \forall\,\, n\in\mathbb{N},
\]
hence obtaining a solution $f\in [\overline{\text{span}}(M_{\Lambda})]_{w,A}$ to the moment problem. The proof is now complete.

\appendix

\section{Proof of Theorem $\ref{spectralsynthesis}$ on a class of operators that admit Spectral Synthesis}

In order to prove Theorem $\ref{spectralsynthesis}$, we need the following result obtained by Markus.

\begin{thmc}\label{Compact}\cite[Theorem 4.1]{Markus1970}

Let $\cal{H}$ be a separable Hilbert space and let $T:\cal{H}$$\to\cal{H}$ be a compact operator such that

(i) its kernel is trivial and

(ii) its non-zero eigenvalues are $\bf simple$.

Let $\{f_n\}_{n\in\mathbb{N}}$ be the corresponding sequence of eigenvectors.
Then $T$ admits $\bf{Spectral\,\, Synthesis}$ if and only if $\{f_n\}_{n\in\mathbb{N}}$ is $\bf Hereditarily\,\, complete$ in $\cal{H}$.
\end{thmc}

We will show in Lemma $\ref{Eigenvalues}$ that
if $\{u_n\}_{n=1}^{\infty}$ is a sequence of real numbers satisfying $(\ref{un})$, then
\begin{equation}\label{operator}
Tf(x) : = \sum_{n=1}^{\infty} \langle f , r_n\rangle_{w, A}\cdot u_n\cdot  x^{\lambda_n},
\end{equation}
is an operator well defined on $[\overline{\text{span}}(M_{\Lambda})]_{w, A}$ to $[\overline{\text{span}}(M_{\Lambda})]_{w, A}$,
such that it is compact, with a trivial kernel, having non-zero simple eigenvalues $\{ u_k\}_{k=1}^{\infty}$
and with $\{ x^{\lambda_k}\}_{k=1}^{\infty}$ being the corresponding eigenvectors.
But since the system
$M_{\Lambda}$ is hereditarily complete in $[\overline{\text{span}}(M_{\Lambda})]_{w, A}$, it will then follow from Theorem $C$ that
$T$ admits Spectral Synthesis and the proof of Theorem $\ref{spectralsynthesis}$ will finish.

\begin{lemma}\label{Eigenvalues}
The following are true about $T$.
\begin{enumerate}
\item $T$ is a well defined bounded operator from $[\overline{\text{span}}(M_{\Lambda})]_{w, A}$ to
$[\overline{\text{span}}(M_{\Lambda})]_{w, A}$ and moreover $T$ is compact.
\item $\{ u_k\}_{k=1}^{\infty}$ are eigenvalues  of $T$ and $\{ e_k\}_{k=1}^{\infty}$ are the corresponding eigenvectors.
\item $\{u_k\}$ are eigenvalues  of $T^*$ (the adjoint of $T$) and $\{ r_k\}_{k=1}^{\infty}$ are the corresponding eigenvectors.
\item The kernel of $T$ is trivial.
\item The spectrum of $T$ is $\displaystyle{\{ 0\} \cup \{ u_k\}_{k=1}^{\infty}}$.
\item Each eigenvalue of $T$ is simple.
\item The operator $T$ is not normal.
\end{enumerate}
\end{lemma}

\begin{proof}

$\quad$

1. Let $f$ be a function in $[\overline{\text{span}}(M_{\Lambda})]_{w, A}$. Then $f$ extends analytically in the slit disk $D_{r_w}$
\[
f(z)=\sum_{n=1}^{\infty} \langle f , r_n\rangle_{w, A}\cdot  z^{\lambda_n},
\]
with the series converging uniformly on compact subsets of $D_{r_w}$.
From relation $(\ref{rnkbound})$, for every $\epsilon>0$ there exists some $m_{\epsilon}>0$, independent of $n\in\mathbb{N}$, so that
\begin{equation}\label{aaa1}
|\langle f , r_n\rangle_{w, A}|\le ||f||_{L^2_w (A)}\cdot ||r_n||_{L^2_w (A)}
\le ||f||_{L^2_w (A)}\cdot m_{\epsilon}(r_w-\epsilon)^{-\lambda_n}.
\end{equation}
Since $\{u_n\}_{n=1}^{\infty}$ satisfies $(\ref{un})$, then
for some $\rho\in (0, 1)$ one has $|u_n|\le \rho^{\lambda_n}$, hence $u_n\to 0$ as $n\to\infty$.

Choose
\begin{equation}\label{epsilonn}
\epsilon=\frac{r_w(1-\rho)}{2},\qquad \text{thus}\qquad r_w-\epsilon=\frac{r_w(1+\rho)}{2}.
\end{equation}
Combining $(\ref{aaa1})-(\ref{epsilonn})$ yields that there exists some $m_{\epsilon}>0$, independent of $n\in\mathbb{N}$ and $f\in L^2_w (A)$,
so that

\begin{equation}\label{upperbound}
|\langle f , r_n\rangle_{w, A}\cdot u_n|\le ||f||_{L^2_w (A)}\cdot m_{\epsilon}\cdot \left(\frac{2\rho}{r_w(1+\rho)}\right)^{\lambda_n}.
\end{equation}

One deduces from $(\ref{upperbound})$ that $Tf(z)$
is a function analytic on the slit disk
\[
\left\{z\in\mathbb{C}\setminus (-\infty, 0]: \,\, |z|<r_w\cdot \left(\frac{1+\rho}{2\rho}\right)\right\}
\]
converging uniformly on its compact subsets, and in particular on the interval $[0,r_w]$.
The uniform convergence on $[0,r_w]$ and the inequality $m\le w(x)\le M$ on the set $A$, imply that

$(i)$ $Tf(z)$ belongs to the space $[\overline{\text{span}}(M_{\Lambda})]_{w, A}$.

$(ii)$ The series $T(f)$ converges in the $L^2_w (A)$ norm.

It also follows from $(\ref{upperbound})$ and $(\ref{operator})$ that there exists some $N>0$ so that
\[
||T(f)||_{L^2_w (A)}\le N||f||_{L^2_w (A)}\quad \text{for\,\, all}\quad f\in [\overline{\text{span}}(M_{\Lambda})]_{w, A}.
\]
Therefore, $T: [\overline{\text{span}}(M_{\Lambda})]_{w, A}\to [\overline{\text{span}}(M_{\Lambda})]_{w, A}$ defines a Bounded Linear Operator.
We will use $\| T \|$ to
denote the operator norm of $T$ which is the supremum of the set
\[
\{ \| T(f)\|_{L^2_w (A)}: f\in [\overline{\text{span}}(M_{\Lambda})]_{w, A}, \| f\|_{L^2_w (A)}=1 \}.
\]
We also denote by $T^*$ the Adjoint operator of $T$.\\

\smallskip

Next we show that $T$ is compact.
Let $e_n(x)=x^{\lambda_n}$.
Define  $T_m$ on $[\overline{\text{span}}(M_{\Lambda})]_{w, A}$ by
\[
T_m(g)(x)=\sum_{n=1}^m \langle g, r_n \rangle_{w, A} u_n e_n(x).
\]
Let $f$ be a unit vector in $[\overline{\text{span}}(M_{\Lambda})]_{w, A}$. It is easy to see that
\[
\|(T-T_m)(f)\|_{L^2_w (A)}\leq \sum_{n=m+1}^\infty \| \langle f, r_n \rangle_{w, A} u_n e_n  \|_{L^2_w (A)}.
\]

It is easy to see that there is some $q>0$ so that one has $\| e_n \|_{L^2_w (A)}<q r_w^{\lambda_n}$. Combined with $(\ref{upperbound})$ and
with $\epsilon$ as in $(\ref{epsilonn})$ gives

\[
\| \langle f, r_n \rangle_{w, A} u_n e_n  \|_{L^2_w (A)}\le m_\epsilon ||f||_{L^2_w (A)}\left( \frac{2\rho}{1+\rho}\right)^{\lambda_n}.
\]

Therefore,
\[
\|(T-T_m)(f)\|_{L^2_w (A)}\le m_\epsilon ||f||_{L^2_w (A)} \sum_{n=m+1}^\infty \left(\frac{2\rho}{1+\rho}\right)^{\lambda_n}.
\]
Since $2\rho<1+\rho$ the series above converges. Hence, $\|T-T_m\|_{L^2_w (A)}$ tends to zero as $m$ tends to infinity.
Thus, the finite rank operators $\{ T_m\}$ converge to $T$ in the uniform operator topology. Therefore,
$T$ is compact.

2. The families $M_{\Lambda}$ and $r_{\Lambda}$ are biorthogonal thus from $(\ref{operator})$ we get
\[
T(e_k)=u_k e_k.
\]

3. For fixed $k\in\mathbb{N}$ and any $n\in\mathbb{N}$ we have
\begin{eqnarray}
\langle T^*r_k- u_k r_k, e_n\rangle_{w, A} & = & - u_k\langle r_k, e_n\rangle_{w, A} +\langle r_k, Te_n\rangle_{w, A}\nonumber\\
& = & - u_k\langle r_k, e_n\rangle_{w, A} + \langle r_k, u_n e_n\rangle_{w, A}\nonumber\\
& = & - u_k\langle r_k, e_n\rangle_{w, A} + u_n \langle r_k, e_n\rangle_{w, A}. \label{zero}
\end{eqnarray}
For $n=k$ $(\ref{zero})$ equals zero, and the same holds for all $n\not= k$
due biorthogonality.
Therefore $\langle T^*r_k - u_k r_k, e_n\rangle_{w, A}=0$ for all $n\in\mathbb{N}$, hence
\[
T^*r_k - u_k\cdot r_k=\mathbf{0}.
\]

4. For any $n\in\mathbb{N}$ we have
\[
\langle Tf, r_n\rangle_{w, A} =\langle f, T^*r_n\rangle_{w, A} = \langle f, u_n r_n\rangle_{w, A}.
\]
If $Tf=0$ then $\langle f, r_n\rangle_{w, A} =0$ for all $n\in\mathbb{N}$. The completeness of the family $\{r_n\}_{n=1}^{\infty}$
in the space $[\overline{\text{span}}(M_{\Lambda})]_{w, A}$ means that $f=\mathbf{0}$. Hence the kernel of $T$ is the zero function.\\

\smallskip

5. Suppose now that $Tf=\lambda f$ for some $\lambda\notin\{u_n\}_{n=1}^{\infty}$ and $f\not=\mathbf{0}$. Then,
\begin{align}
\lambda\langle f, r_n\rangle_{w, A} & =  \langle Tf, r_n\rangle_{w, A}, \notag \\
& =\langle f, T^*r_n\rangle_{w, A} \notag, \\
&= \langle f, u_n r_n\rangle_{w, A}, \notag \\
&=u_n\langle f, r_n\rangle_{w, A}.\label{Only_eigenvalues}
\end{align}
Thus,
\[
(\lambda-u_n)\cdot \langle f, r_n\rangle_{w, A}=0 \quad \text{for all}\quad  n\in\mathbb{N}.
\]
Since $\lambda\notin\{u_n\}_{n=1}^{\infty}$ then $\langle f, r_n\rangle_{w, A}=0$ for all $n\in\mathbb{N}$,
and the completeness of the family $\{r_n\}_{n=1}^{\infty}$ in $[\overline{\text{span}}(M_{\Lambda})]_{w, A}$
means that $f=\mathbf{0}$. We conclude that the $u_k$'s are the only non-zero eigenvalues of $T$ and form its
compactness it follows that the spectrum of $T$ is
\[
\displaystyle{\{ 0\} \cup \{ u_k\}_{k=1}^{\infty}}.
\]

6. Suppose now that $Tf=u_k f$ for some $u_k\in\{u_n\}_{n=1}^{\infty}$ and some $f\in [\overline{\text{span}}(M_{\Lambda})]_{w, A}$.
Then, using the same computation as in (\ref{Only_eigenvalues}) above, we get

\[
(u_k-u_n)\langle f , r_n\rangle_{w, A} = 0 \quad \text{for all}\quad n\in\mathbb{N}.
\]
But  $u_n\not= u_k$ if $n\not= k$, thus $\langle f , r_n\rangle_{w, A} = 0$ for all $n\not= k$. Since
\[
f(x)=\sum_{n=1}^{\infty} \langle f , r_n\rangle_{w, A}\cdot  e_n(x),
\]
we get $f(t)=\langle f , r_k\rangle_{w, A} e_k$,  meaning that every $u_k$ is simple.\\

\smallskip

7. Finally, suppose that $T$ is a normal operator, thus
\[
TT^*(e_k)=T^*T(e_k)\qquad \text{for\,\,all}\quad k\in\mathbb{N}.
\]
Any two eigenvectors that correspond to different eigenvalues of a normal operator are orthogonal.
Clearly the set of eigenvectors $e_n=x^{\lambda_n}$ of $T$ is not orthogonal, therefore $T$ is not normal.

This concludes the proof of Lemma $\ref{Eigenvalues}$.
\end{proof}

\end{document}